\documentclass{article}

\usepackage[utf8]{inputenc}
\usepackage{tabularx}
\usepackage{subcaption}
\usepackage{enumitem}
\usepackage[margin=1in]{geometry} 
\usepackage{amsmath,amsthm,amssymb}
\usepackage{bbm}
\usepackage{dsfont}
\usepackage{xcolor}
\usepackage{graphicx}
\usepackage{tikz}
\usetikzlibrary{shapes.geometric, arrows}
\usepackage[
backend=biber,
style=numeric,
maxbibnames=99, minbibnames=99
]{biblatex}
\usepackage[colorlinks]{hyperref}
\hypersetup{
    colorlinks=true,
    linkcolor=blue,
    filecolor=magenta,      
    urlcolor=cyan,
    citecolor=red
}

\def\R{{\mathbb R}} 
\def\N{{\mathbb N}} 
\def\W{{\mathcal W}} 

{}
\def\N{\mathbb N}

\def\W{\mathcal W}

\def\bb1{\mathds 1}

\theoremstyle{plain}
\newtheorem{theorem}{Theorem}[section]
\newtheorem{proposition}[theorem]{Proposition}
\newtheorem{lemma}[theorem]{Lemma}
\newtheorem{corollary}[theorem]{Corollary}

\theoremstyle{definition}
\newtheorem{definition}[theorem]{Definition}

\theoremstyle{remark}
\newtheorem{remark}[theorem]{Remark}

\newtheorem{question}{Question}
\numberwithin{subcase}{case}
\title{Cut norm discontinuity of triangular truncation of graphons}
\author{Teddy Mishura}
\date{\today}
\begin{document}
\maketitle

\abstract{The space of $L^p$ graphons, symmetric measurable functions $w: [0,1]^2 \to \R$ with finite $p$-norm, features heavily in the study of sparse graph limit theory. We show that the triangular cut operator $M_{\chi}$ acting on this space is not continuous with respect to the cut norm. This is achieved by showing that as $n\to \infty$, the norm of the triangular truncation operator $\mathcal{T}_n$ on symmetric matrices equipped with the cut norm grows to infinity as well. Due to the density of symmetric matrices in the space of $L^p$ graphons, the norm growth of $\mathcal{T}_n$ generalizes to the unboundedness of $M_{\chi}$. We also show that the norm of $\mathcal{T}_n$ grows to infinity on symmetric matrices equipped with the operator norm.}
\section{Introduction}
The theory of graph limits, introduced in \cite{LOVASZ2006933}, is based upon the idea that a series of convergent (dense) graphs will have limit objects in the form of symmetric measurable functions from $[0,1]^2$ to $[0,1]$. These limit objects are known as graphons and play a large role in modeling and analyzing large networks, showing up in fields such as large deviations \cite{chatterjee2011,lubetzky2015}, mathematical statistics \cite{abbe2017,borgs2012}, and random graph theory \cite{bollosbas2007,diaconis2008}. In general, it is more helpful to consider the adjacency matrices of the graphs that converge to a graphon instead, and the tool of choice to use with these objects is the cut norm, defined in \cite{friezekannan}. The cut norm can also be defined for functions in general (see \cite[Definition 8.13]{lovaszbook}), and in fact convergence in cut norm was shown in \cite{LOVASZ2006933} to be the correct way to view convergence of graph sequences. When the edge density of a sequence of graphs tends to 0, dense graph limit theory states that the limit object must be the zero graphon. However, many interesting structures in applications are derived from graphs whose edge density becomes 0, such as regimes governed by a power law degree distribution. Borgs et~al. developed a sparse graph theory in \cite{Borgs_2019,Borgs_2018} where they consider graphons that live in $L^p([0,1]^2)$, whose space we denote $\mathcal{W}^p$. The cut norm remains the proper way to consider convergence of these sequences.

An interesting subclass of graphons are the Robinson graphons, symmetric functions $w$ on the unit square such that for $0 \leq x \leq y \leq z \leq 1$, 
\begin{equation*}
    w(x,z) \leq \min(w(x,y),w(y,z)).
\end{equation*}
Such functions were defined as ``diagonally increasing'' in \cite{Chuangpishit_2015}, as they increase if you move toward the diagonal along a horizontal or vertical line. These functions are generalizations of Robinson matrices, sometimes called R-matrices, which appear in the classic problem of seriation (see \cite{Liiv_2010} for a comprehensive review of seriation) and whose study is a field of much interest \cite{chepoi2009,flammarion2016optimal,fogel2014,fortin2014,laurent2017,}. It can be easily seen that being Robinson is a strict property; it is often the case that a matrix of interest is ``almost'' Robinson except for relatively few entries. However, such a matrix can be viewed as being sampled from some underlying Robinson graphon, allowing approximate results to be generated via analysis of the graphon instead of the matrix. A prevalent issue for such approximations is measuring how ``close'' a matrix (or graphon) is to being Robinson to inform how accurate the results from the prior analysis must be.

One's first thought for this problem is to simply compare every entry in the matrix with its preceding neighbors and add up the ``errors''; clearly, you will only be Robinson if you end up with $0$. In \cite{janssen2019optimization}, they defined a parameter $\Gamma_1$ that does this, and while $\Gamma_1$ is easy to compute, it fails to be continuous in cut norm
(or equivalently the graph limit topology). Thus $\Gamma_1$ is not a suitable Robinson measurement for growing networks as it does not respect limits of graph sequences. In \cite{Chuangpishit_2015}, the authors defined a new parameter $\Gamma$ which is $0$ if its argument is Robinson; they also showed it is continuous with respect to the cut norm, providing a tool to measure Robinson resemblance of large graphs. However, this is only for the dense case; their proof of continuity for $\Gamma$ relies upon the fact that the operator multiplying a graphon by 1 above the main diagonal and by 0 below the main diagonal is bounded in cut norm for graphons in $\W^{\infty}$. However, they were unable to show if this method of proof could be extended to the sparse case, as it was unknown if that operator was bounded in cut norm for graphons in $\W^p$. To answer this question, we can look toward matrices for inspiration, seeing how a similar ``triangular cut'' would behave for them in cut norm. As symmetric matrices form a dense subspace of graphons, we can prove results in the world of matrices to then make statements in the world of graphons.

In this paper, we study the norm of the triangular cut on matrices with respect to the cut norm. It is not unusual to see the cut norm appear in a study of operators; while the cut norm arose from the graph theoretical problem of finding the maximal ``cut'' of a given weighted graph, it has been of much use in more analytical pursuits, such as approximation results for graphons \cite{lovasz2007} and algorithmic approximations for matrix problems \cite{alon2004,friezekannan}. Consider an $n \times n$ matrix and let the operator $\mathcal{T}_n$ set all entries of that matrix below the main diagonal to 0. $\mathcal{T}_n$ is called the triangular cut (also known as triangular truncation), and its norm growth is a well-studied problem in operator theory; see for example \cite{ANGELOS1992117,Zhou_2016} for explicit calculations and bounds for $\mathcal{T}_n$ applied to matrices equipped with operator norm. 

Outside of the cut norm, it is well known that the operator norm of $\mathcal{T}_n$ grows to infinity when it is considered as an operator on real matrices equipped with the standard operator norm $\|\cdot\|_{\text{opr}}$ (see \cite[Theorem 1]{ANGELOS1992117} for a proof and \cite{Watson1995} for estimations of growth speed). In \cite{bennett1976}, Bennett showed that when $1<p<q<\infty$, the triangular cut mapping $\ell^p$ to $\ell^q$ is bounded.  Recently, Coine used the canonical characterization of Schur multipliers -- an operator that multiplies one matrix element-wise with another matrix -- to prove that the triangular cut mapping $\ell^p$ to $\ell^q$ is unbounded when $1\leq q \leq p\leq\infty$ \cite{coine}. 

However, to make conclusive statements about graphons we must know how the norm of $\mathcal{T}_n$ grows on symmetric matrices. Existing literature lacked any such analysis, with respect to the cut norm or otherwise, and so we address this issue here. To do so, we view $\mathcal{T}_n$ as a Schur multiplier and further note that the cut norm is equivalent to an injective tensor norm. This allows us to make use of some bounds and techniques from \cite{bennett1977} to show that the norm of $\mathcal{T}_n$ on symmetric matrices grows to infinity. We refer to \cite{bennett1977,coine,davidson} for results on the norms of Schur multipliers in general. Due to the author's interest, we also present a proof of the unboundedness of $\mathcal{T}_n$ on symmetric matrices with respect to the standard operator norm.
\section{Notation and background}\label{appendix1}
In this section we shall provide the necessary background needed to understand the graph limit theory and matrix norms used throughout the paper, as well as provide proofs for several inequalities that feature here. 

We begin with a quick overview of the theory of $L^p$ sparse graph limits, developed in \cite{Borgs_2019,Borgs_2018}. Dense graph limit theory is concerned primarily with graphons and kernels, as defined in \cite[Chapter 7]{lovaszbook}- graphons being symmetric measurable functions from $[0,1]^2 \to [0,1]$ and kernels being linear combinations of graphons. The space of graphons is denoted $\W_0$. In the sparse theory, we shall work with the generalization of kernels known as $L^p$ graphons, as sparse limit objects do not need to be bounded.
\begin{definition}[$L^p$ graphons]\label{def:lpgraphons}
Let $1\leq p \leq \infty$. An $L^p$ \textit{graphon} is a symmetric, measurable function $w \in L^p([0,1]^2)$. We refer to the space of $L^p$ graphons as $\W^p$. In the case $p=\infty$, the space $\W^\infty$ is just the usual kernels as defined above. 
\end{definition}
For each $p$, $\W^p$ is a linear space, and by H\"older's inequality, for $p < q$, we have $\W^q \subseteq \W^p$.  
We equip $\W^p$ with the cut norm, denoted by $\|\cdot\|_{\Box}$, as defined below.
For every $w\in \W^1$, 
\begin{equation}\label{def:cut-norm}
    \|w\|_{\square} = \sup_{S,T \subseteq [0,1]}\Bigg|\int_{S \times T}w(x,y)dxdy\Bigg|,
\end{equation}
where the supremum is taken over all measurable subsets $S$ and $T$. 
Clearly, $\|\cdot\|_\Box\leq \|\cdot\|_1$, but the two norms are not equivalent. Every $w\in \W^1$ may be viewed as the kernel of an integral operator $T_w:L^\infty[0,1]\to L^1[0,1]$ defined as $T_w(f)(x)=\int_{[0,1]} w(x,y)f(y) dy$, and the two norms $\|w\|_\Box$ and 
$\|T_w\|_{\text{opr}}$ are equivalent (see \cite[Lemma 8.11]{lovaszbook}). 
From the functional analytic point of view, the norm $\|T_w\|_{\text{opr}}$ is just the injective tensor norm in 
$L^1[0,1]\check{\otimes} L^1[0,1]$. 
Now we shall define the matrix norms used throughout the paper as well as provide a proof of several important properties these norms possess.
\begin{definition}\label{def:matrixnorm}
Let $A$ be an $n \times n$ matrix and let $[n]=\{1,...,n\}$. We denote the \textit{cut norm} of $A$, introduced in \cite{friezekannan}, by 
\begin{equation*}
    \|A\|_{\Box} = \frac{1}{n^2}\Bigg|\sup_{A,B \subseteq [n]}\sum_{\substack{i \in A \\ j \in B}} a_{ij}\Bigg|,
\end{equation*}
and the $(p,q)$-norm of $A$ \cite[Equation 1]{bennett1977} by
\begin{equation*}
    \|A\|_{\mathcal{B}(\ell^p,\ell^q)} = \sup_{\|x\|_p \leq 1} \|Ax\|_q,
\end{equation*}
where $\|\cdot\|_p$ is the standard $p$-norm. We use the more conventional notation $\|A\|_{\text{opr}}$ to denote the $(2,2)$-norm, which is simply called the operator norm of $A$.
\end{definition}

As our proof of Proposition \ref{prop:unboundedtn} relies heavily on the relation between these two norms and tensor products of matrices, which we denote with the symbol $\otimes$, we state and make use of the following properties.
\begin{proposition} \label{prop:normineq}
For $n \times n$ matrices $A,B$, we have
\begin{enumerate}[label=(\alph*)]
\item \label{lem:cutnormopnorm} $n\|A\|_{\Box} \leq \|A\|_{\text{opr}}$. 
\item\label{prop:normineq1} $n^2\|A\|_{\Box} \leq \|A\|_{\mathcal{B}(\ell^\infty,\ell^1)} \leq 4n^2\|A\|_{\Box}$.
\item\label{prop:normineq2} $\|A\|_{\mathcal{B}(\ell^\infty,\ell^1)}\|B\|_{\mathcal{B}(\ell^\infty,\ell^1)} \leq \|A \otimes B\|_{\mathcal{B}(\ell^\infty,\ell^1)} \leq \frac{\pi}{2}K_G^2\|A\|_{\mathcal{B}(\ell^\infty,\ell^1)}\|B\|_{\mathcal{B}(\ell^\infty,\ell^1)}$, where $1.676 < K_G < 1.782$ is the Grothendieck constant. 
\end{enumerate}
\end{proposition}

\begin{proof}
A concise proof of \ref{lem:cutnormopnorm} can be found in \cite{nikiforov2009cutnorms}. The upper bound of \ref{prop:normineq2} follows from Proposition 10.2 of \cite{bennett1977} while the lower bound is known. \ref{prop:normineq1} is a discretization of Lemma 8.11 in \cite{lovaszbook}; this adds a factor of $n^2$, as seen below.
\begin{equation*}
    \|A\|_{\mathcal{B}(\ell^\infty,\ell^1)} = \sup_{\|x\|_{\infty} \leq 1} \|Ax\|_1 = \sup_{\|x\|_{\infty}\leq 1}\sum_{i = 1}^n\Big|\sum_{j=1}^n a_{ij}x_j\Big| = \sup_{\|x\|_{\infty},\|y\|_{\infty}\leq 1}\sum_{i = 1}^n\sum_{j=1}^n a_{ij}x_jy_i,
\end{equation*}
and we refer to \cite[Lemma 8.10]{lovaszbook} for an equivalent definition of the cut norm of $A$: 
$$\displaystyle \|A\|_{\Box} = \frac{1}{n^2}\sup_{0 \leq x_i,y_j \leq 1} \Big|\sum_{i,j}a_{ij}x_iy_j\Big|,$$ 
showing that the lower bound holds. For the upper bound, we rewrite $\|A\|_{\mathcal{B}(\ell^\infty,\ell^1)}$ in the following way:
\begin{align*}
    \|A\|_{\mathcal{B}(\ell^\infty,\ell^1)} &= \sup_{0 \leq x_i,z_i,y_j,w_j \leq 1} \Big|\sum_{i,j}a_{ij}(x_i-z_i)(y_j-w_j)\Big| \\
    &\leq \sup_{0 \leq x_i,z_i,y_j,w_j \leq 1} \Big|\sum_{i,j}a_{ij}x_iy_j\Big|+\Big|\sum_{i,j}a_{ij}x_iw_j\Big|+\Big|\sum_{i,j}a_{ij}z_iy_j\Big|+\Big|\sum_{i,j}a_{ij}z_iw_j\Big| \\
    &= 4n^2\|A\|_{\Box},
\end{align*}
proving the original claim. 
\end{proof}
\section{The triangular cut}\label{section:triangle}
In this section, we show that the linear operator $M_{\chi}$ defined by the triangular cut on graphons is not continuous (i.e.~not bounded). Prefacing our result, we need the following notations and definitions.
\begin{definition}[Triangular cut on matrices]\label{def:trian-matrix}
Let $M_n({\mathbb R})$ denote the space of real matrices of size $n$, and $\mathcal{S}_n(\mathbb R)$ denote the subspace of $M_n(\mathbb R)$ consisting of symmetric matrices. Define $T_n\in M_n(\mathbb R)$ as 
\begin{equation}\label{eq:trimat}
(T_n)_{ij} = \begin{cases}
                1 & i \geq j \\
                0 & i < j
            \end{cases},
\end{equation}
and define the triangular cut for matrices as 
\begin{equation}\label{eq:triopr-n}
\mathcal{T}_n:  M_n(\mathbb R)\rightarrow  M_n(\mathbb R), \quad A \mapsto A \circ T_n,
\end{equation}
where $\circ$ represents Schur multiplication (i.e.~entrywise multiplication).
\end{definition}
\begin{definition}[Triangular cut on graphons]\label{def:trian-graphon}
Let $\chi:[0,1]^2\to [0,1]$ be defined as
\begin{equation}
    \chi(x,y) = \begin{cases}
                1 & x \leq y \\
                0 & x > y
            \end{cases},
\end{equation}
and define the triangular cut on graphons as
\begin{equation}\label{eq:mchi}
    M_{\chi} : \W^p \to L^p([0,1]^2), \quad w \mapsto w\chi.
\end{equation}
\end{definition}
It is of note that $\|M_{\chi}\|_{\text{opr}} \geq \|\mathcal{T}_n\|_{\text{opr}}$ for all $n$, as for any matrix $M \in M_n(\R)$, $\|\mathcal{T}_n(M)\|_{\Box} \leq \|M_{\chi}(w_M)\|_{\Box}$, where $w_M$ is an $n \times n$ step function whose entries are identical to $M$ but doubled on the diagonal. We now state and prove our main result, noting that Proposition \ref{prop:unboundedtn} (iii) is included for sake of completion, as the author could not find a similar result in the literature.
\begin{proposition}\label{prop:unboundedtn}
For $\mathcal{T}_n$ and $T_n$ as in Definition \ref{def:trian-matrix}, we have:
\begin{enumerate}
    \item[(i)] $\sup_{A \in M_n({\mathbb R})} \dfrac{\|\mathcal{T}_n(A)\|_{\Box}}{\|A\|_{\Box}} \to \infty$ as $n \to \infty$.
    
    \item[(ii)] $\sup_{A \in \mathcal{S}_n({\mathbb R})} \dfrac{\|\mathcal{T}_n(A)\|_{\Box}}{\|A\|_{\Box}} \to \infty$ as $n \to \infty$.
    
    \item[(iii)] $\sup_{A \in \mathcal{S}_n({\mathbb R})} \dfrac{\|\mathcal{T}_n(A)\|_{\text{opr}}}{\|A\|_{\text{opr}}} \to \infty$ as $n \to \infty$.
\end{enumerate}
\end{proposition}

\begin{proof}
Consider the $n \times n$ matrix $A_n$ defined as below:
\begin{equation}\label{eq:an-matrix}
    (A_n)_{ij} = \begin{cases}
    0 & i = j \\
    (i-j)^{-1} & i \neq j
    \end{cases}.
\end{equation}
By Proposition \ref{prop:normineq} \ref{lem:cutnormopnorm}, we know that $n\|A_n\|_{\Box} \leq \|A_n\|_{\text{opr}}$. To bound $\|A_n\|_{\text{opr}}$, one can represent $A_n$ as the ``upper corner'' of a Toeplitz operator $T$ on $\ell^2$ given by $T(e_n) = \sum_{k=1}^{n-1}\frac{1}{n-k}e_k + \sum_{k=n+1}^{\infty}\frac{1}{n-k}e_k$, where $\{e_n\}_n$ forms the canonical basis of $\ell^2(\N)$. $T$ can be viewed as the operator $T_f$ on $H^2$ (the space of holomorphic functions on the complex unit disk with finite 2-norm), where $f(e^{i\theta})= i(\pi-\theta)$, and a standard theroem of Toeplitz operators \cite[p. 179]{Douglas1998} can be used to show that $\|T_f\|_{\text{opr}}= \sup|f(e^{i\theta})| = \pi$. Thus, $\|T\|_{\text{opr}}= \pi$, showing that $\|A_n\|_{\text{opr}} \leq \pi$, and so it must be that $\|A_n\|_{\Box} \leq \frac{\pi}{n}$. Calculating $\|\mathcal{T}_n(A_n)\|_{\Box}$ is simple, as the matrix is strictly positive above the diagonal, so we get that 
\begin{equation}\label{eq:boxnormtan}
    \|\mathcal{T}_n(A_n)\|_{\Box} = \frac{1}{n^2}\sum_{1 \leq i<j \leq n} (A_n)_{ij} = \frac{1}{n^2}\sum_{i = 0}^{n-1} \frac{i}{(n-i)} =\frac{n(H_{n-1}-1)+1}{n^2},
\end{equation}
where $H_n$ is the $n$-th harmonic number. This last equality can be shown through induction; it clearly holds true for $n=2$, and we note that 
\begin{equation*}
    \sum_{i=0}^{n} \frac{i}{n+1-i} = \sum_{i = 0}^{n-1} \frac{i}{(n-i)} +H_n 
    = n(H_{n-1}-1)+1 +H_n 
    = (n+1)H_n-n,
\end{equation*}
proving the claim. Thus we have 
\begin{equation*}
    \sup_{A \in M_n(\R)} \frac{\|\mathcal{T}_n(A)\|_{\Box}}{\|A\|_{\Box}} \geq \frac{\|\mathcal{T}_n(A_n)\|_{\Box}}{\|A_n\|_{\Box}} \geq \frac{n^2(H_{n-1}-1)+n}{n^2} \to \infty,
\end{equation*}
proving (i).

To prove (ii), consider again $A_n$ as defined in \eqref{eq:an-matrix}, and note that $A_n \otimes A_n$ is a symmetric matrix.
Using  Proposition \ref{prop:normineq} \ref{prop:normineq1} and \ref{prop:normineq2}, we have
\begin{eqnarray*}
    n^4\|A_n \otimes A_n\|_{\Box} &\leq& \|A_n \otimes A_n\|_{\mathcal{B}(\ell^\infty,\ell^1)} 
    \ \leq\  \frac{\pi}{2}K_G^2\|A_n\|_{\mathcal{B}(\ell^\infty,\ell^1)}\|A_n\|_{\mathcal{B}(\ell^\infty,\ell^1)} \\
    &\leq& 2\pi K_G^2n^2\|A_n\|_{\Box}\|A_n\|_{\mathcal{B}(\ell^\infty,\ell^1)} 
   \  \leq\  2\pi^2 K_G^2n\|A_n\|_{\mathcal{B}(\ell^\infty,\ell^1)}.
\end{eqnarray*}
On the other hand, by Proposition \ref{prop:normineq} \ref{prop:normineq1},
\begin{eqnarray*}
    4n^4\|T_{n^2} \circ (A_n \otimes A_n)\|_{\Box} &\geq& \|T_{n^2} \circ (A_n \otimes A_n)\|_{\mathcal{B}(\ell^\infty,\ell^1)} 
    = \|(T_n \circ A_n) \otimes A_n\|_{\mathcal{B}(\ell^\infty,\ell^1)} \\
    &\geq& \|T_n \circ A_n\|_{\mathcal{B}(\ell^\infty,\ell^1)}\|A_n\|_{\mathcal{B}(\ell^\infty,\ell^1)} 
    \geq n^2\|T_n \circ A_n\|_{\Box}\|A_n\|_{\mathcal{B}(\ell^\infty,\ell^1)} \\
    &=& (n(H_{n-1}-1)+1)\|A_n\|_{\mathcal{B}(\ell^\infty,\ell^1)},
\end{eqnarray*}
where the last equality was shown in \eqref{eq:boxnormtan}. 
Thus, 
\begin{equation*}
    \sup_{A \in \mathcal{S}_{n^2}} \dfrac{\|\mathcal{T}_{n^2}(A)\|_{\Box}}{\|A\|_{\Box}} \geq \frac{\|\mathcal{T}_{n^2}(A_n \otimes A_n)\|_{\Box}}{\|A_n\otimes A_n\|_{\Box}} \geq \frac{n(H_{n-1}-1)+1}{8\pi^2K_G^2n} \to \infty,
\end{equation*}
proving (ii). 

To prove (iii), we make use of the well-known fact that for any matrices $A,B$, $\|A \otimes B\|_{\text{opr}} = \|A\|_{\text{opr}}\|B\|_{\text{opr}}$, as well as once again considering the symmetric matrix $A_n \otimes A_n$. It is clear that $\|A_n \otimes A_n\|_{\text{opr}} = \|A_n\|_{\text{opr}}\|A_n\|_{\text{opr}} \leq \pi\|A_n\|_{\text{opr}}$, where $\|A_n\|_{\text{opr}} \leq \pi$ by the same Fourier analytic argument used in the proof of (i). On the other hand,
\begin{equation*}
    \|T_{n^2} \circ (A_n \otimes A_n)\|_{\text{opr}} = \|(T_n \circ A_n) \otimes A_n\|_{\text{opr}} 
    = \|T_n \circ A_n\|_{\text{opr}}\|A_n\|_{\text{opr}} 
    \geq \frac{4}{5} \log n\|A_n\|_{\text{opr}},
\end{equation*}
where $\|T_n \circ A_n\|_{\text{opr}} \geq \frac{4}{5} \log n$ can be shown by considering $T_n \circ A_n$ applied to the vector $(n-1)^{-1/2}\sum_{k=2}^ne_k.$ Thus we have 
\begin{equation*}
    \sup_{A \in \mathcal{S}_{n^2}} \dfrac{\|\mathcal{T}_{n^2}(A)\|_{\text{opr}}}{\|A\|_{\text{opr}}} \geq \frac{\|\mathcal{T}_{n^2}(A_n \otimes A_n)\|_{\text{opr}}}{\|A_n\otimes A_n\|_{\text{opr}}} \geq \frac{4}{5\pi} \log n \to \infty,
\end{equation*}
proving (iii).
\end{proof}
This result can easily be shown to imply unboundedness of $M_{\chi}$ on bounded graphons.
\begin{corollary}\label{cor:mchi-unbounded}
Let $M_{\chi}$ be as defined in \eqref{eq:mchi}. Then for all $c > 0$, there exists a graphon $w\in \W^{\infty}$ such that $\|M_{\chi}(w)\|_{\Box} > c\|w\|_{\Box}.$ In other words, $\|M_{\chi}\|_{\text{opr}} = \infty$ on the space $(\W^{\infty},\|\cdot\|_{\Box})$. Furthermore, as $(\W^{\infty},\|\cdot\|_{\Box}) \subset (\W^p,\|\cdot\|_{\Box})$ for all $1 \leq p < \infty$, $M_{\chi}$ acting on ($\W^p,\|\cdot\|_{\Box})$ is also an unbounded operator.
\end{corollary}

\begin{proof}
Let $c > 0$ be fixed. Then by Proposition \ref{prop:unboundedtn} (ii) there exists an $n \in \N$ such that $\|\mathcal{T}_{n^2}(A_n\otimes A_n)\|_{\Box} > c\|A_n\otimes A_n\|_{\Box}$, where $A_n$ is the matrix defined in \eqref{eq:an-matrix}. Let $I_i = [\frac{i-1}{n^2},\frac{i}{n^2})$ for $1 \leq i \leq n^2$ and define the graphon $w_n$ as
\begin{equation}\label{eq:tensor-graphon}
    w_n(x,y) = \begin{cases}
    (A_n\otimes A_n)_{ij} & (x,y) \in I_i \times I_j,\quad i\neq j \\
    0 & (x,y) \in I_i \times I_i
    \end{cases}
\end{equation}
It is clear that $\|M_{\chi}(w_n)\|_{\Box} =\|\mathcal{T}_{n^2}(A_n\otimes A_n)\|_{\Box}$ and that $\|w_n\|_{\Box} = \|A_n\otimes A_n\|_{\Box}$. Thus, it must be that $\|M_{\chi}(w_n)\|_{\Box} > c\|w_n\|_{\Box}$, proving the claim and showing that $M_{\chi}$ is unbounded on $(\W^{\infty},\|\cdot\|_{\Box}).$ Furthermore, as $w_n \in \W^p$ for all $p \geq 1$, it is clear that $M_{\chi}$ is also unbounded on $(\W^p,\|\cdot\|_{\Box})$ for all such $p$.
\end{proof}

\begin{remark}
In \cite{Chuangpishit_2015}, the authors prove that for graphons $w$ such that $w:[0,1]^2\to [-2,2]$, $\|w\chi\|_{\Box} \leq 2\sqrt{\|w\|_{\Box}}$. However, we must be careful to note that their result does not imply continuity of $M_{\chi}$ on $(\W^{\infty},\|\cdot\|_{\Box})$, as this would imply that the square root could be ``dropped''. Clearly this cannot happen due to Corollary \ref{cor:mchi-unbounded}, and so we make note that this result can only be extended to bounded subsets of $\W^{\infty}$. This is because following the authors in \cite{Chuangpishit_2015}'s proof, we get that $\|w\chi\|_{\Box} \leq \|w\|_{\infty}\sqrt{\|w\|_{\Box}}$, showing that this technique fails for $\W^{\infty}$ in general.
\end{remark}
\section{The banded cut}
The triangular cut sends matrices to their upper triangular part, a subclass of matrices that are of much mathematical interest. Another family of matrices that receive attention are the banded matrices: these are matrices $A$ such that there exists a constant $k \in \N$ where $A_{i,j} = 0$ if $|i-j|>k$. Banded matrices are useful in fast computations of numerical approximations, such as the finite element method as applied to certain partial differential equations \cite[Section 1.9]{hughes2000}. In this section, we use our results on the triangular cut on graphons to show that the ``banded'' cut on graphons is also unbounded. We begin by defining the banded cut below.
\begin{definition}[Banded cut on graphons]\label{def:banded-oprs}
Let $0 < \lambda < 1$ and let the function $\chi_{\lambda}:[0,1]^2\to [0,1]$ be defined as 
\begin{equation*}
    \chi_{\lambda}(x,y)= \begin{cases}
    1 & |x-y| \leq 1-\lambda\\
    0 & \text{else},
    \end{cases}
\end{equation*}
and define the banded cut on graphons as 
\begin{equation*}
    B_{\lambda}: \W^p \to \W^p, \quad w\mapsto w\chi_{\lambda}.
\end{equation*}
\end{definition}
One can view this cut as snipping off two triangles of equal size from opposite edges of the unit square. We now introduce a lemma that will prove useful for our proof of the unboundedness of the banded cut.

\begin{figure}[h]
\centering
    \includegraphics[width=0.35\textwidth]{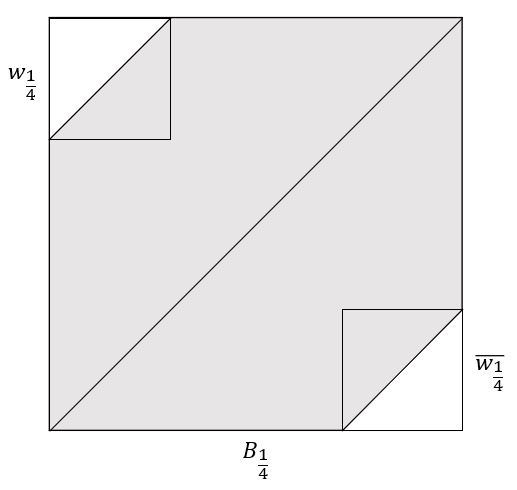}
    \caption{The operator $B_{\frac{1}{4}}$ acting on the function $w_{\frac{1}{4}}+\overline{w}_{\frac{1}{4}}$.}\label{fig:fig3}
\end{figure}

\begin{lemma}\label{lem:cutnormtrans}
Let $w \in L^p([0,1]^2)$ and for $0 < \lambda \leq \frac{1}{2}$, define the two functions $w_{\lambda},\overline{w}_{\lambda} \in L^p([0,1]^2)$ as
\begin{align*}
    w_{\lambda}(x,y) &= \begin{cases}
    w(x/\lambda, y/\lambda - (1-\lambda)/\lambda) & (x,y) \in [0,\lambda]\times[1-\lambda,1]\\
    0 & \text{else}
    \end{cases}\\
    \overline{w}_{\lambda}(x,y) &= w_{\lambda}(y,x).
\end{align*}
Then $\|w_{\lambda}+\overline{w}_{\lambda}\|_{\Box} = 2\lambda^2\|w\|_{\Box}$.
\end{lemma}

\begin{proof}
We first consider the following upper bound
\begin{align*}
    \|w_{\lambda}+\overline{w}_{\lambda}\|_{\Box} &\leq \|w_{\lambda}\|_{\Box}+\|\overline{w}_{\lambda}\|_{\Box} \\
    &= \lambda^2\|w\|_{\Box}+\lambda^2\|w\|_{\Box} \\
    &= 2\lambda^2\|w\|_{\Box},
\end{align*}
noting that this holds true as $w_{\lambda}$ and $\overline{w}_{\lambda}$ are just copies of $w$ ``shrunken'' by a factor of $(1-\lambda)^2$. For a lower bound, let $S_1 \times T_1 \subseteq [0,\lambda]\times[1-\lambda,1]$ and let $S_2 \times T_2 \subseteq [1-\lambda,1]\times [0,\lambda].$ We make note of the fact that
\begin{align}\label{eq:flipflop-lowerbnd}
    \|w_{\lambda}+\overline{w}_{\lambda}\|_{\Box} &\geq \bigg|\iint_{(S_1\cup S_2)\times(T_1\cup T_2)}w_{\lambda}+\overline{w}_{\lambda}\bigg| \nonumber\\
    &= \bigg|\iint_{S_1\times T_1}w_{\lambda}+\overline{w}_{\lambda} + \iint_{S_1\times T_2}w_{\lambda}+\overline{w}_{\lambda} + \iint_{S_2\times T_1}w_{\lambda}+\overline{w}_{\lambda} + \iint_{S_2\times T_2}w_{\lambda}+\overline{w}_{\lambda}\bigg| \nonumber\\
    &= \bigg|\iint_{S_1\times T_1}w_{\lambda}+\overline{w}_{\lambda}+\iint_{S_2\times T_2}w_{\lambda}+\overline{w}_{\lambda}\bigg| = \bigg|\iint_{S_1\times T_1}w_{\lambda}+\iint_{S_2\times T_2}\overline{w}_{\lambda}\bigg|,
\end{align}
as $w_{\lambda}+\overline{w}_{\lambda} = 0$ on both $S_1 \times T_2$ and $S_2 \times T_1$, $w_{\lambda} =0$ on $S_2\times T_2$, and $\overline{w}_{\lambda} = 0$ on $S_1\times T_1$. Letting $S_1 \times T_1$ and $S_2 \times T_2$ be the sets on which $\|w_{\lambda}\|_{\Box}$ and $\|\overline{w}_{\lambda}\|_{\Box}$ are respectively achieved, and noting that $\|w_{\lambda}\|_{\Box}=\|\overline{w}_{\lambda}\|_{\Box} = \lambda^2\|w\|_{\Box}$, \eqref{eq:flipflop-lowerbnd} implies that 
\begin{equation*}
    \|w_{\lambda}+\overline{w}_{\lambda}\|_{\Box} \geq \|w_{\lambda}\|_{\Box}+\|\overline{w}_{\lambda}\|_{\Box} =2\lambda^2\|w\|_{\Box},
\end{equation*}
proving the original claim.
\end{proof}
\begin{corollary}\label{cor:banded-unbndd}
Let $1 \leq p \leq \infty$. Then for $0 < \lambda \leq \frac{1}{2}$, $B_{\lambda}$ is an unbounded operator on $(\W^p,\|\cdot\|_{\Box})$. In other words, for any $c > 0$, there exists a graphon $w \in \W^p$ such that $\|B_{\lambda}(w)\|_{\Box} > c\|w\|_{\Box}$. 
\end{corollary}
The idea of the following proof is to make use of the fact that the triangular cut is unbounded on graphons by forcing the banded cut to act like the triangular cut on two graphons at once. For a visual explanation of how this is done, refer to Figure \ref{fig:fig3}.
\begin{proof}
Let $0 < \lambda \leq \frac{1}{2}$ be fixed, let $n \geq 2$, let $w_n$ be the graphon defined in \eqref{eq:tensor-graphon}, and let $A_n$ be the matrix defined in \eqref{eq:an-matrix}. We will now consider the behavior of $B_{\lambda}$ applied to the graphon $(w_n)_{\lambda}+(\overline{w_n})_{\lambda}$, making liberal use of Lemma \ref{lem:cutnormtrans}. We first note that
\begin{equation*}
    \|(w_n)_{\lambda}+(\overline{w_n})_{\lambda}\|_{\Box} = 2\lambda^2\|w_n\|_{\Box} = 2\lambda^2\|A_n \otimes A_n\|_{\Box} \leq \frac{4\lambda^2\pi^2K_G^2}{n^3}\|A_n\|_{\mathcal{B}(\ell^{\infty},\ell^1)},
\end{equation*}
where the last inequality was shown in Proposition \ref{prop:unboundedtn}. Furthermore, we also note that 
\begin{align*}
    \|B_{\lambda}((w_n)_{\lambda}+(\overline{w_n})_{\lambda})\|_{\Box} &= \|(w_n\chi)_{\lambda}+(\overline{w_n\chi})_{\lambda}\|_{\Box} = 2\lambda^2\|w_n\chi\|_{\Box} \\
    &=2\lambda^2\|A_n\otimes A_n\|_{\Box} \geq \frac{\lambda^2(n(H_{n-1}-1)+1)}{2n^4}\|A_n\|_{\mathcal{B}(\ell^{\infty},\ell^1)},
\end{align*}
where the last inequality was also shown in Proposition \ref{prop:unboundedtn}. Thus, we can show that
\begin{equation*}
    \sup_{w \in \W^p} \frac{\|B_{\lambda}(w)\|_{\Box}}{\|w\|_{\Box}} \geq \frac{\|B_{\lambda}(w_n)\|_{\Box}}{\|w_n\|_{\Box}} \geq \frac{n(H_{n-1}-1)+1}{8\pi^2K_G^2n} \to \infty,
\end{equation*}
showing that for $1 \leq p \leq \infty$ and $0 < \lambda \leq \frac{1}{2}$, $\|B_{\lambda}\|_{\text{opr}} = \infty$ on $(\W^p,\|\cdot\|_{\Box})$, proving the corollary.
\end{proof}

In terms of graphs, one can think of the operator $B_{\lambda}$ as deleting a proportion of the edges of a graph equal to $\lambda^2$. Alternatively, one could imagine a banded cut where one removed triangle is larger/smaller than the other. Such an operator could be viewed as deletion of some edges while directing some others, and a similar argument could be used to show the unboundedness of this operation as well. 
\section{Applications to graph limits}

Proposition \ref{prop:unboundedtn} (and Corollary \ref{cor:mchi-unbounded}) can be used to show many interesting results in the realm of graph limits, of which we list a few. We first start by defining some key concepts integral to the study of graph limits.
\begin{definition}[Associated graphon]
Let $G$ be a labelled graph on $n$ vertices and let $A_G$ be the associated adjacency matrix of $G$. Then, the \textit{associated graphon} of $G$ is denoted $w_G$ and defined as:
\begin{equation*}
    w_G(x,y) := (A_G)_{ij} \quad (x,y) \in I_i \times I_j,
\end{equation*}
where $I_i = (\frac{i-1}{n},\frac{i}{n}]$ for $1 \leq i \leq n$.
\end{definition}

\begin{definition}[Convergence of graphs]
A sequence of (dense) labelled graphs $\{G_n\}_{n\geq 1}$ is said to \textit{converge} to a graphon $w \in \W^{\infty}$ if 
\begin{equation*}
    \lim_{n \to \infty}\|w_{G_n}-w\|_{\Box} = 0. 
\end{equation*}
However, if $\|w_{G_n}\|_1/n^2 \to 0$, then the limit of such a sequence will always be the zero graphon. We call such a sequence of graphs \textit{sparse}. To provide a meaningful limit object when considering sparse graph sequences, we normalize each associated graphon by its $L^1$ norm before considering convergence (introduced in \cite{Borgs_2019,Borgs_2018}), noting that a limit object for such a sequence must necessarily not be bounded. Thus, a sparse sequence of labelled graphs $\{H_n\}_{n\geq 1}$ is said to \textit{converge} to a graphon $w \in \W^p$ if 
\begin{equation*}
    \lim_{n \to \infty}\|w_{H_n}/\|w_{H_n}\|_1 - w\|_{\Box} = 0.
\end{equation*}
\end{definition}

\begin{figure}[t]
    \centering 
    \begin{minipage}[b]{0.45\textwidth}
        \centering
        \includegraphics[width=0.7\textwidth]{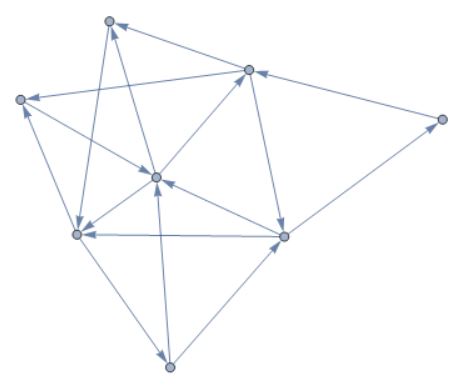} 
    \end{minipage}
    \begin{minipage}[b]{0.45\textwidth}
        \centering
        \includegraphics[width=0.55\textwidth]{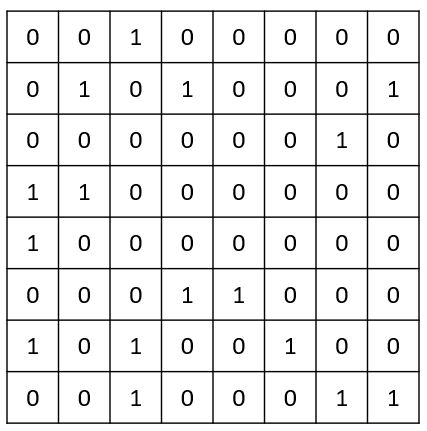}
    \end{minipage}
    \caption{An example of a directed graph $G$ and its associated graphon $w_G$.}\label{fig:fig1}
\end{figure}

The triangular cut applies to graph limits in a way that relates to directed graphs. Limits of sequences of directed graphs, dense or sparse, can be represented as measurable functions $w: [0,1]^2 \to \R$, with the caveat that such functions no longer need to be symmetric. An example of a directed graph and its associated graphon can be found in Figure \ref{fig:fig1}. Directing a graph $G$ by turning an edge $i \sim j$ into $i \to j$ for $i > j$ will act on $A_G$ exactly as the triangular cut would. For an example of such a transformation of a graph, see Figure \ref{fig:fig2}. 

Directing graphs in this way will always be a bounded operation with respect to the cut norm, as associated graphons of simple graphs are always nonnegative and thus for any such $w_G$, we have that $\|w_G\chi\|_{\Box} \leq \frac{1}{2}\|w_G\|_{\Box}.$ Furthermore, for dense graph sequences, this direction scheme does not affect convergence. For any simple graphs $G,H$, we have that $\|(w_G-w_H)\chi\|_{\Box} \leq \sqrt{\|w_G-w_H\|_{\Box}}$ (\cite[Lemma 6.2]{Chuangpishit_2015}), and so convergence of $\{w_{G_n}\}$ implies convergence of $\{w_{G_n}\chi\}$.

However, such results do not hold for sequences of sparse graphs. Given a convergent sequence of sparse graphs $\{G_n\}$, as $\|w_{G_n}\|_1 \to 0$, it must necessarily be the case that $\sup_n \|w_{G_n}/\|w_{G_n}\|_1\|_{\infty} = \infty$. Therefore the upper bound given in \cite{Chuangpishit_2015} is useless when considering the convergence of $\{w_{G_n}\chi\}$, and by Corollary \ref{cor:mchi-unbounded} we have that $\|(w_{G_n}/\|w_{G_n}\|_1-w_{G_m}/\|w_{G_m}\|_1)\chi\|_{\Box}$ can be arbitrarily large compared to $\|w_{G_n}/\|w_{G_n}\|_1-w_{G_m}/\|w_{G_m}\|_1\|_{\Box}$. Thus it is not the case that direction in this fashion is a limit respecting operation for sparse graph sequences. Interestingly, this type of direction is equivalent to making directed acyclic graphs \cite[Proposition 1.4.3]{bang2008digraphs}, and so its potential to affect convergence seems to be a matter of importance.

We shall close the paper with a few open questions of interest to the author.
\begin{question}
Does there exist a sequence of graphons $\{w_n\}_{n\geq 1} \subset \W^p$ such that $w_n$ is convergent to some $w \in \W^p$ but $\{w_n\chi\}_{n \geq 1}$ is not Cauchy with respect to the cut norm?
\end{question}
\begin{question}
Corollary \ref{cor:mchi-unbounded} shows that ``cutting'' a graphon by a triangle is an unbounded operator with respect to the cut norm. Can cut norm bounded operators of such form (i.e. multiplying by a characteristic function of some set) be characterized?
\end{question}

\begin{figure}[h]
    \centering 
    \begin{minipage}[b]{0.45\textwidth}
        \centering
        \includegraphics[width=0.7\textwidth]{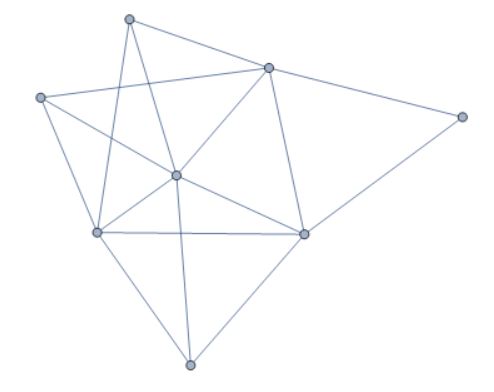} 
    \end{minipage}
    \begin{minipage}[b]{0.45\textwidth}
        \centering
        \includegraphics[width=0.7\textwidth]{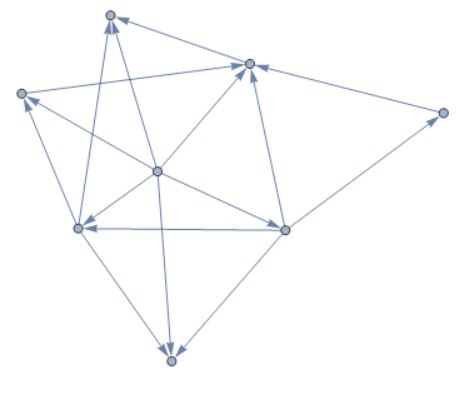}
    \end{minipage}
    \caption{An example of a graph $G$ and a direction assignment corresponding to the triangular cut.}\label{fig:fig2}
\end{figure}
\section{Acknowledgements}
The author would like to acknowledge Mahya Ghandehari and Charli Klein for their help in editing this paper and the Department of Mathematical Sciences at the University of Delaware for their continued support throughout the process of this research.
\printbibliography
\end{document}